\documentclass[10pt]{amsart}

\usepackage{amssymb}
\usepackage{enumitem}
\usepackage[bottom]{footmisc}
\usepackage[usenames,dvipsnames]{xcolor}
\usepackage[bookmarks,colorlinks=true,citecolor=OliveGreen,linkcolor=RoyalBlue]{hyperref}
\usepackage{mathtools}
\usepackage{nicefrac}
\usepackage{randtext}

\usepackage{pgfplots}
\usetikzlibrary{arrows.meta}
\usetikzlibrary{cd}

\usepackage[capitalize]{cleveref}

\title{Cousin's lemma in second-order arithmetic}

\author[J.M. Barrett]{Jordan Mitchell Barrett}
\address{Victoria University of Wellington, NZ}
\email{\randomize{math@jmbarrett.nz}}
\urladdr{https://jmbarrett.nz/}

\author[R.G. Downey]{Rodney G. Downey}
\address{Victoria University of Wellington, NZ}
\email{\randomize{rod.downey@vuw.ac.nz}}
\urladdr{https://homepages.ecs.vuw.ac.nz/~downey/}

\author[N. Greenberg]{Noam Greenberg}
\address{Victoria University of Wellington, NZ}
\email{\randomize{greenberg@msor.vuw.ac.nz}}
\urladdr{https://homepages.ecs.vuw.ac.nz/~greenberg/}

\date{\today}

\thanks{Downey and Greenberg were partially supported by the Marsden Fund of New Zealand. Many of the results in this paper are also contained in Barrett's honours thesis \cite{barrettThesis}.}

\theoremstyle{definition}
\newtheorem{theorem}{Theorem}[section]
\newtheorem{lemma}[theorem]{Lemma}

\newtheorem{definition}[theorem]{Definition}

\newtheorem{proposition}[theorem]{Proposition}

\newtheorem{remark}[theorem]{Remark}

\makeatletter
\def\thmhead@plain#1#2#3{%
	\thmname{#1}\thmnumber{\@ifnotempty{#1}{ }\@upn{#2}}%
	\thmnote{ {\the\thm@notefont#3}}}
\let\thmhead\thmhead@plain
\makeatother

\setlist[enumerate,1]{label={(\roman*)}}
\setlist[enumerate,2]{label={(\alph*)}}
\setlist[enumerate]{itemsep=1pt}
\setlist[itemize]{itemsep=0pt}
\setlist{topsep=5pt}

\newcommand{\N}{\mathbb{N}}			
\newcommand{\R}{\mathbb{R}}			

\newcommand{\K}{\mathcal{K}}

\newcommand{\RCA}{\mathsf{RCA}_0}		
\newcommand{\WKL}{\mathsf{WKL}_0}		
\newcommand{\ACA}{\mathsf{ACA}_0}		
\newcommand{\ATR}{\mathsf{ATR}_0}		

\newcommand{\abs}[1]{{\left\lvert #1 \right\rvert}}		

\newcommand{\seq}[1]{{\left\langle{#1}\right\rangle}}
\newcommand{\Nat}{\mathbb N}

\renewcommand{\epsilon}{\varepsilon}
\newcommand{\s}{\sigma}

\newcommand{\andd}{\,\,\,\&\,\,\,}

\newcommand{\w}{\omega}

\newcommand{\rest}[1]{\,\upharpoonright\,#1}

\newcommand{\vphi}{\varphi}
\newcommand{\conc}{\hat{\,\,}}

\newcommand{\arxiv}[1]{\url{https://arxiv.org/abs/#1}}

\begin{document}

\begin{abstract}
    Cousin's lemma is a compactness principle that naturally arises when studying the gauge integral, a generalisation of the Lebesgue integral.	We study the axiomatic strength of Cousin's lemma for various classes of functions, using Friedman and Simpson's reverse mathematics in second-order arithmetic. We prove that, over $\RCA$:
    \begin{enumerate}
    	\item Cousin's lemma for continuous functions is equivalent to $\WKL$;
    	
    	\item Cousin's lemma for Baire class 1 functions is equivalent to $\ACA$;
    	
    	\item Cousin's lemma for Baire class 2 functions, or for Borel functions, are both equivalent to $\ATR$ (modulo some induction). 
    \end{enumerate}
\end{abstract}

\maketitle

\section{Introduction}

For real-valued functions, the most general version of the fundamental theorem of calculus relies on a generalization of both the Lebesgue and (improper) Riemann integrals, called, among other names, the Denjoy integral \cite{denjoyExtensionIntegraleLebesgue1912,gordonIntegralsLebesgueDenjoy1994}. This integral can integrate all derivatives, which the Riemann and Lebesgue integrals fail to do \cite{gordonIntegralsLebesgueDenjoy1994}.

Denjoy \cite{denjoyExtensionIntegraleLebesgue1912} first defined this integral in 1912, and shortly after, Luzin \cite{luzinProprietesIntegraleDenjoy1912} and Perron \cite{perronUberIntegralbegriff1914} gave equivalent characterisations. However, all these definitions were complex and highly nonconstructive, making Denjoy's integral impractical for applications \cite{gordonNonabsoluteIntegrationWorth1996}.
In 1957, Kurzweil \cite{kurzweilGeneralizedOrdinaryDifferential1} discovered an equivalent, ``elementary'' definition, in the style of the Riemann integral. This re-formulation is widely known as the \textit{gauge integral}, or Henstock--Kurzweil integral.

For bounded functions with compact support, the gauge integral coincides with the Lebesgue integral, and hence the gauge integral gives a Riemann-like definition of the Lebesgue integral a for wide class of functions. We also remark that the gauge integral is suitable for integrating classes of 
highly discontinuous functions, and can be viewed as a mathematically rigorous formalisation of Feynman's path integral \cite{muldowney}.

A \emph{gauge} is simply a function $\delta\colon [0,1]\to \R^+$ from the unit interval to the positive real numbers. A \emph{$\delta$-fine partition} is a finite tagged partition $0=x_0 \le \xi_0 \le x_1 \le \xi_1 \le \dots \le x_n =1$ for which $\delta(\xi_i)\ge  x_{i+1}-x_i$. This generalises the mesh size of a partition: a partition has mesh size $\le \delta$ if it is $\delta$-fine for the constant map with value~$\delta$. The gauge integral is defined exactly like the Riemann integral, except that it allows arbitrary gauges rather than only constant ones: the integral $\int_0^1f\,dx$ is $K$ if for every $\epsilon>0$ there is a gauge~$\delta$ such that for every $\delta$-fine partition, the associated Riemann sum is within~$\epsilon$ of~$K$. So every Riemann integrable function is gauge integrable, by choosing constant gauges. But also the function $x^{-\nicefrac{1}{2}}$ is gauge integrable on~$[0,1]$, by taking, given $\epsilon>0$, a gauge satisfying $\delta(p)\le \epsilon p^{-\nicefrac{1}{2}}$ for $p>0$. And Dirichlet's function (the indicator function of the rationals) is integrable by taking $\delta(q_n) = \epsilon 2^{-n}$, where $\seq{q_n}$ lists the rationals in $[0,1]$.

\medskip

The key fact that enables the theory of the gauge integral is \emph{Cousin's lemma} \cite{cousinFonctionsVariablesComplexes1895}, which states that for any gauge~$\delta$, a $\delta$-fine partition exists. 
Without Cousin's lemma, Kurzweil's definition of the gauge integral would be vacuously satisfied by all functions and all values, trivialising the notion.  

\smallskip

The gauge integral and Cousin's lemma have had metamathematical exploration before; notably
in the setting of descriptive set theory by Dougherty and Kechris \cite{DK91}, Becker \cite{Bec92}, and Walsh \cite{WalshDefin}, for example. Recently, it was also explored in higher-order reverse mathematics by Normann and Sanders \cite{sandersNormannCousin,MR4182786}. These studies found that the gauge integral requires powerful axioms to prove. For example, working in third-order arithmetic, Normann and Sanders showed that Cousin's lemma and the existence of the gauge integral were each equivalent to full second-order arithmetic.

The point is that Cousin's lemma is a statement about \emph{all} functions from $[0,1]$ to $\R^+$; no definability assumptions are involved. It is thus a statement which is not expressible is second-order arithmetic, which only allows quantification over real numbers (but not sets of real numbers). Second-order arithmetic is the study of countable mathematics, including objects which can be coded by countable objects, for instance separable metric spaces, continuous functions, and Borel sets. 

Reverse mathematics of second-order arithmetic, as developed by H.~Friedman and S.~Simpson, is the project of understanding the proof-theoretic strength of the theorems of mathematics in terms of comprehension and induction strength required to prove them. What axioms of mathematics are actually required to prove a particular statement? The aim is to find optimal proofs, ones which require the least ``amount'' of axioms; and then show that these proofs are indeed the best that can be found, by finding a ``reversal'': a proof of the axioms used, starting with the investigated theorem. 

The main insight of the project of reverse mathematics is that almost all theorems of mainstream mathematics are equivalent to one of five axiomatic systems, and that these systems are linearly ordered by logical implication; see, for example~\cite{simpsonSubsystemsSecondOrder2009}. One advantage of working within second-order arithmetic is that the proof-theoretic strength is often aligned with complexity, as defined using the tools of computability theory. Very informally speaking, a theorem which asserts the existence of complicated objects, requires strong axioms to prove, and vice-versa. This connection between computability and proof theory has resulted in a rich body of research. 

\smallskip

It is thus natural to investigate the strength of Cousin's lemma within second-order arithmetic, but in order to do so, we must restrict ourselves to classes of countably-coded functions. In this paper we consider Borel functions, and sub-classes of these. 

Below we will observe that Cousin's lemma is a form of a compactness principle. We would thus naturally start by looking at the system $\WKL$ (weak K\"onig's lemma), which informally is understood to be equivalent to the compactness of the unit interval. But more specifically, it is equivalent to \emph{countable compactness}, for example, to the statement that any countable open cover of $[0,1]$ has a finite sub-cover \cite[Thm.IV.1.2]{simpsonSubsystemsSecondOrder2009}. For continuous functions, Cousin's lemma should be equivalent to its version for constant functions, but this itself uses a form of compactness, namely that continuous functions on $[0,1]$ obtain a minimum. Nonetheless, we are able to show:
\begin{itemize}
    \item Cousin's lemma for continuous gauges is equivalent to $\WKL$. 
\end{itemize}

However, once we go up the hierarchy of Borel functions (as measured say by the Baire class of a function), it turns out that this form of uncountable compactness is significantly stronger:
\begin{itemize}
    \item Cousin's lemma for Baire class 1 gauges is equivalent to $\ACA$. 
    \item Cousin's lemma for all Borel gauges is provable in $\ATR+\Delta^1_2$-induction. 
    \item Cousin's lemma for Baire class 2 gauges implies $\ATR$. 
\end{itemize}

We leave open the question of whether the extra induction is required.

\section{Cousin's lemma and compactness} 

Let $\delta\colon[0,1]\to \R^+$ be a gauge. We can view~$\delta$ as assigning, to each~$x\in [0,1]$, the open interval $(x-\delta(x), x+\delta(x))$. Cousin's lemma is very close to saying that this open cover of the unit interval has a finite sub-cover. It will be easier to work with this simplified version, rather than with tagged paritions, partly because this notion can be extended to other compact metric spaces; we will be working with Cantor space as well as the unit interval. 

The metamathematical twist is that for our reversals, we will be working over $\RCA$, in which case some of the gauges that we will work with will not have values in the model. For example, we will work with Baire class 1 functions, coded in the model by a sequence of continuous functions, where the values of the functions are Cauchy sequences, which need not converge. Nonetheless, relations such as $f(x)> r$ will be still definable in the model, and this will be enough to formalise Cousin's lemma. 

In light of this, it is important to note that a $\delta$-fine tagged parition $0= x_0 \le \xi_0 \le x_1 \le \xi_1 \le \dots \le x_{n-1}\le \xi_{n-1}\le x_n = 1$, with $x_{i+1}-x_i \le \delta(\xi_i)$ gives us not only the finite set of points $\xi_i$ giving us a cover, say $[0,1]\subseteq \bigcup_i (\xi_i-2\delta(\xi_i), \xi_i + 2\delta(\xi_i))$, but also the distances $x_{i+1}-x_i$, bounded by $\delta(x_i)$; and while $\delta(\xi_i)$ may not be in the model, the value $x_{i+1}-x_i$ is. Thus we define:

\begin{definition} \label{def:delta_fine_cover}
    Let $X$ be a compact metric space, and let $\delta\colon X\to \R^+$ be a gauge on~$X$. A \emph{$\delta$-fine cover} of~$X$ is a finite set $P\subseteq X$ and a function $p\mapsto r_p$ from $P$ to~$\R^+$ satisfying:  
    \begin{enumerate}
        \item $X = \bigcup_{p\in P} B(p,r_p)$; and
        \item For all $p\in P$, $r_p \le \delta(p)$. 
    \end{enumerate}
\end{definition}

Assuming that $\delta$ is a function coded in the model for which the relations $\delta(x)\in B$ (for an open or closed ball~$B$) are defined, we have:

\begin{lemma}[($\RCA$)] \label{lem:delta_fine_partition_and_cover}
    Let $\delta\colon [0,1]\to \R^+$ be a gauge. 
    \begin{enumerate}
        \item[(a)] If there is a $\delta$-fine tagged partition then there is a $2\delta$-fine cover. 
        \item[(b)] If there is a $\delta$-fine cover then there is a $2\delta$-fine tagged partition. 
    \end{enumerate}
\end{lemma}

\begin{proof}
    For~(a), suppose that $0=x_0 \le \xi_0 \le x_1 \le \xi_1 \le x_2 \le \dots \le \xi_{n-1} \le x_n=1$ is a $\delta$-fine partition: for all $i<n$, $\delta(\xi_i)\ge x_{i+1}-x_i$. Let $P = \{\xi_i\,:\, i<n\}$, and for $i< n$ let $r_{\xi_i} = 2(x_{i+1}-x_i)$.  Then for all $i<n$, $r_{\xi_i} \le 2\delta(\xi_i)$ and trivially $[x_i,x_{i+1}]\subset (\xi_i-r_{\xi_i}, \xi_i + r_{\xi_i})$, so $P, \bar r_p$ is a $2\delta$-fine cover.

    \smallskip
    
    For~(b), suppose that $P = \left\{ p_0 < p_1 < \dots < p_{n-1}    \right\}$ equipped with $i\mapsto r_{p_i}$ is a $\delta$-fine cover. For brevity, for $p\in P$ let $U(p) = (p-r_p, p+r_p)$. By applying a reverse greedy algorithm, we may assume that~$P$ is minimal in that for no distinct $p,q$ in~$P$ do we have $U(p)\subseteq U(q)$. This implies that for all $i<n-1$, $p_{i+1}-p_i < r_{p_i} + r_{p_{i+1}}$; otherwise, a point between~$p_i$ and~$p_{i+1}$ would not be covered by $U(p_i)$ or by $U(p_{i+1})$; if $x\in U(p_j)$ then either $U(p_i)\subseteq U(p_j)$ (if $j<i$) or $U(p_{i+1})\subseteq U(p_j)$ (if $j>i+1$). So we choose 
     \[
     x_{i+1}\in (p_i,p_{i+1}) \cap U(p_i) \cap U(p_{i+1}),
     \]
    as well as $x_0 = 0$ and $x_n=1$; then $x_0 \le p_0 \le x_1 \le p_1 \le \dots$ is a $2\delta$-fine partition, as $[x_i,x_{i+1}]\subseteq U(p_i)$. 
\end{proof}

In light of \cref{lem:delta_fine_partition_and_cover}, we state:
\begin{definition}[(Cousin's lemma)] \label{def:Cousin}
    Let~$X$ be a compact metric space, and let~$\K$ be a class of functions, Cousin's lemma for~$\K$ on~$X$ states that every $\delta\colon X\to \R^+$ in~$\K$ has a $\delta$-fine cover. 
\end{definition}

If not stated otherwise, we work with $X = [0,1]$. For many of our arguments, though, it will be more convenient to work in Cantor space, under the metric $d(x,y) = 2^{-n}$ for the largest~$n$ satisfying $x\rest{n} = y\rest{n}$. We state one connection between Cousin's lemma on Cantor space and on the unit interval informally. Once we work with particular classes of functions, we will see that the following argument, in each particular case, can be made to hold in $\RCA$.

\begin{lemma} \label{lem:Cantor_to_UI}
    For reasonable classes $\K$ of functions, Cousin's lemma for~$\K$ on Cantor space implies Cousin's lemma for~$\K$ on the unit interval. 
\end{lemma}

\begin{proof}
    We use the standard map $\vphi\colon 2^\w\to [0,1]$ defined by 
    \[
        \vphi(x) = \sum x(n)2^{-n-1}. 
    \]
    Then~$\vphi$ is continuous and for all $x,y\in 2^\w$, 
    \[
        |\vphi(x)-\vphi(y)| \le d(x,y). 
    \]
    Hence, given $\delta\colon [0,1]\to \R^+$ in~$\K$, we let $\hat \delta = \vphi\circ \delta$, which is a gauge on Cantor space. Suppose that $(\hat P,\bar r_p)$ is a $\hat\delta$-fine cover. Let $P = \left\{ \vphi(p) \,:\,  p\in \hat P \right\}$, and let $r_{\vphi(p)} = r_p$; this is a $\delta$-fine cover. 
\end{proof}

\section{The Borel case} 

The ``classical'' proof of Cousin's lemma can be carried out in the system of $\ATR+\Sigma^1_1$-induction. Note that while we can formalise Borel codes and functions in $\RCA$, in the system $\ATR$ we can prove that for every Borel function~$f$ and every~$x$, there is a point $f(x)$. 

We need the following:

\begin{lemma}[($\ATR+\Sigma^1_1$-induction)] \label{lem:tree_lemma}
    If $T\subseteq 2^{<\w}$ is an infinite $\Pi^1_1$-definable tree, then~$T$ has a path. 
\end{lemma}

\begin{proof}
    Let 
    \[
            S = \left\{ \s\in T \,:\,  (\exists^\infty \tau\succeq \s)\,\,\tau\in T \right\}. 
        \]    
    Then~$S$ is $\Pi^1_1$ (this uses $\Sigma^1_1$-choice, which is provable in $\ATR$; see \cite[VIII.3.21]{simpsonSubsystemsSecondOrder2009}). By assumption, the empty sequence is in~$S$; and every $\s\in S$ has a child $\s\conc i$ in~$S$. So we can apply $\Pi^1_1$-dependent choice on numbers, which is provable in  $\ATR+\Sigma^1_1$-induction, \cite[VIII.4.10]{simpsonSubsystemsSecondOrder2009}. 
\end{proof}

\begin{proposition}[($\ATR+\Delta^1_2$-induction)] \label{prop:Cousin_for_Borel}
    Cousin's lemma holds for all Borel functions. 
\end{proposition}

\begin{proof}
    By \cref{lem:Cantor_to_UI}, we work with Cantor space. Let $\delta\colon 2^\w\to \R^+$ be a Borel function. Let 
    \[
        G = \left\{ \s\in 2^{<\w} \,:\,  (\exists x\in [\s])\,\,\, [\s]\subseteq B(x,\delta(x)) \right\}. 
    \]
    Note that $[\s]\subseteq B(x,\delta(x))$ if and only if $\delta(x)> 2^{-|\s|+1}$, so this is a Borel condition. However, the quanitification over all $x\in [\s]$ means that~$G$ is $\Sigma^1_1$ (relative to a code for~$\delta$). 

    Now let 
    \[
        T = \left\{ \s\in 2^{<\w} \,:\,  (\forall \tau\preceq \s)\,\,\tau\notin G \right\}. 
    \]
    Then~$T$ is a $\Pi^1_1$ tree. We claim that it is finite. If not, then by \cref{lem:tree_lemma}, there is some $y\in [T]$. Let~$n$ be sufficiently large so that $2^{-n}< \delta(y)$; this contradicts $y\rest{n}\in T$. 

    \smallskip
    
    So we know that there is some~$n$ which bounds the length of all the strings in~$T$. By $\Delta^1_2$ induction, $T$ exists, as a finite set: by induction on the length-lexicographic order of strings of length~$\le n$, we show that for all $\tau\in 2^{\le n}$, $T\cap \{\s\,:\, \s \text{ precedes }\tau\}$ is a finite set in the model. The formula over which we induct is of the form $(\exists m\le 2^n)(\Pi^1_1\andd \Sigma^1_1)$, so $\Delta^1_2$-induction certainly suffices. 

    It follows that the set of leaves of~$T$ exists (again as a finite set). By taking all immediate extensions of the leaves of~$T$, we obtain a finite antichain $R\subseteq G$ of strings with $2^\w = [R]^\preceq = \bigcup \left\{ [\s] \,:\, \s\in R  \right\}$. 

    Finally, there is a sequence $\seq{x_{\s}\,:\, \s\in R}$ witnessing that each $\s\in R$ is in~$G$, i.e., $x_\s\in [\s]$ and $[\s]\subseteq B(x,\delta(x))$. This follows from $\Sigma^1_1$-choice (or by $\Sigma^1_1$-induction, as~$R$ is finite). Letting $P = \left\{ x_\s \,:\,  \s\in R \right\}$ and $r_{x_\s} = \delta(x_\s)$ gives a $\delta$-fine cover. 
\end{proof}

\section{The continuous case} 

In this section we show that over the base system~$\RCA$ of recursive comprehension, Cousin's lemma for continuous functions is equivalent to weak K\"onig's lemma~$\WKL$. If $\delta$ is continuous then for all~$x$, $\delta(x)$ exists, and so for a $\delta$-fine cover $(P,\bar r_p)$, we can always choose $r_p = \delta(p)$. 

\medskip

In one direction, the argument is quick. 

\begin{proposition}[($\WKL$)] \label{prop:WKL_implies_continuous_Cousin}
    Cousin's lemma holds for continuous functions.
\end{proposition}

\begin{proof}
    Let $\delta\colon [0,1]\to \R^+$ be continuous. Weak K\"onig's lemma implies that~$\delta$ obtains a minimum~$r$, which must be positive. A partition of~$[0,1]$ into intervals of length $2^{-n} < r$, and choosing any points as tags, gives a $\delta$-fine partition.
\end{proof}

Note that the argument works for Cantor space as well. 

\smallskip

In the other direction, we could argue directly by coding a $\Pi^0_1$ class by an effectively closed subset of~$[0,1]$; but the argument is combinatorial simpler by passing via the Heine-Borel theorem. 

\begin{theorem}[($\RCA$)]\label{thm:cl->hb}
    Cousin's lemma for continuous functions implies $\WKL$.
\end{theorem}

\begin{proof}
    $\WKL$ is equivalent to the Heine-Borel theorem, which states the countable compactness of the unit interval: if $\seq{U_n}_{n\in \Nat}$ is a sequence of open intervals and $[0,1]\subseteq \bigcup_n U_n$, then there is finite sub-cover: there is some $N\in \Nat$ such that $[0,1]\subseteq \bigcup_{n<N} U_n$ (see \cite[Lem IV.1.1]{simpsonSubsystemsSecondOrder2009}). 

    Let $\seq{U_n}$ be such an open cover of~$[0,1]$. 

    \medskip
    
    For $n\in \Nat$, let $d_n(x) = d(x,U_n^\complement)$ be the distance from a point~$x$ to the complement of~$U_n$. A code for~$d_n$ as a continuous function can be obtained uniformly from the (end-points of) the interval~$U_n$. Let 
    \[
        \delta(x) = \tfrac{1}{4}\sum_{n\in \Nat} 2^{-n}d_n(x);
    \]
    Then~$\delta$ is a continuous function on~$[0,1]$ (a code can be obtained effectively from the codes for the $d_n$'s). Since $\seq{U_n}$ covers the unit interval, for all~$x$ there is some $n$ such that $d_n(x)>0$, so $\delta$ is positive on~$[0,1]$. 

    \medskip
    
    We claim: for all $p\in [0,1]$ and $k\in \Nat$, 
    \begin{itemize}
        \item[(*)] if $\delta(p)>2^{-k}$ then $(p-\delta(p),p+\delta(p))\subseteq \bigcup_{n\le k} U_n$. 
    \end{itemize}

    Suppose not: then for all $n\le k$, $d_n(p)< \delta(p)$. As a result, 
    \[
        \delta(p) < 
        \tfrac14 \sum_{n\le k}2^{-n}\delta(p) + \tfrac14\sum_{n>k}2^{-n}d_n(p) \le
        \tfrac12 \delta(p) + \tfrac{1}{4} 2^{-k} < \delta(p),
    \]
    which is impossible; of course we use $d_n(p)\le 1$ for all~$p$. 

    \medskip
    
    Now let~$P\subset [0,1]$ be a $\delta$-fine cover (as mentioned we can choose $r_p = \delta(p)$). For each $p\in P$ find some $k$ with $\delta(p) > 2^{-k}$; by $\Sigma_1$ bounding, there is some~$k$ such that for all $p\in  P$, $\delta(p)> 2^{-k}$. Hence $\seq{U_n}_{n\le k}$ covers~$[0,1]$, as required. 
\end{proof}

\begin{remark} 
    The proof above may obscure the main intuition behind it. Consider a model in which $\WKL$ fails, witnessed by a ``cover'' $\seq{U_n}$ which is not really a cover: outside the model, there are points not covered by any~$U_n$, but there are none such inside it. Let $C = [0,1]\setminus \bigcup_n U_n$. This is an effectively closed set (relative to an oracle in the model), and the rough idea is to let $\delta(x) = d(x,C)$. Then no $x\in C$ can be covered by a $\delta$-fine partition. The fact that for any \emph{finite}~$P$, $[0,1]$ is not covered by $\bigcup_{p\in P}(p-\delta(p),p+\delta(p))$ reflects down to the original model. The function~$\delta$ is not really a gauge, but the model doesn't see this, as it thinks that~$C$ is empty. 

    The only difficulty is that $x\mapsto d(x,C)$, while continuous, is not computable relative to an oracle in the model; it is only lower semi-computable. In other words, the model does not contain a code for this function. The definition of~$\delta$ in the proof above gives a continuous lower bound to $d(x,C)$ with a code in the model. 
\end{remark}

\section{Baire class 1} 
\label{sec:baire_class_1}

A function is \emph{Baire class 1} if it is the pointwise limit of continuous functions. The Baire classes were introduced by Baire in his PhD thesis \cite{baireFonctionsVariablesReelles1899}, as a natural generalisation of the continuous functions. One motivation for Baire functions is that many functions arising in analysis are not continuous, such as step functions \cite{heavisideElectromagneticTheory1893}, Walsh functions \cite{walshClosedSetNormal1923}, or Dirichlet's function \cite{dirichletConvergenceSeriesTrigonometriques1829}. However, all such ``natural'' functions generally have low Baire class; for example, the derivative of any differentiable function is Baire class 1, as are functions arising from Fourier series \cite{kechrisClassificationBaireClass1990}.

The Baire classes have previously been studied with respect to computability~\cite{kuyperEffectiveGenericityDifferentiability2014, porterNotesComputableAnalysis2017}. In particular, Kuyper and Terwijn showed that a real number $x$ is 1-generic if and only if every \textit{effective} Baire 1 functions is continuous at $x$ \cite{kuyperEffectiveGenericityDifferentiability2014}. Indeed, the Turing jump function is in some sense a universal Baire class 1 function on Baire space. 

\smallskip

In second-order arithmetic, we can formalise this notion as follows:
\begin{definition} \label{def:Baire_class_1}
    A code for a \emph{Baire class 1} function is a sequence $\seq{f_n}$ of (codes of) continuous functions such that for all~$x$, $\seq{f_n(x)}$ is a Cauchy sequence.
\end{definition}
Note that we do not require the limit of this Cauchy sequence to exist; this would follow from arithmetic comprehension, but the definition still makes sense in $\RCA$. If $\seq{f_n}$ is a code in the model for a Baire class 1 function~$f$, the relation ``$f(x)\in B$'' for an open or closed ball~$B$ is definable, even if $f(x)$ is not an object in the model; for example, we say that $f(x)\in B(y,r)$ if for some $s<r$, for all but finitely many~$n$, $f_n(x)\in B(y,s)$. Thus, it is meaningful to say that $\seq{f_n}$ is a code for a gauge. 

\medskip

In this section we show that Cousin's lemma for Baire class 1 functions is equivalent to~$\ACA$. We start with:

\begin{lemma}[($\RCA$)] \label{lem:Sigma_2_preimages}
    If $f\colon X\to Y$ is Baire class 1, then for open sets $U\subseteq Y$, $f^{-1}[Y]$ is $F_\sigma$, uniformly. 
\end{lemma}

\begin{proof}
    What this means: if $f = \seq{f_n}$ is a code for a Baire class 1 function, and $\seq{U_n}$ is a sequence of (codes of) open subsets of~$Y$, then there is a sequence $\seq{F_n}$ of (codes of) $F_\sigma$ subsets of~$X$ with 
    \[
        F_n = \left\{ x\in X \,:\,  f(x)\in U_n \right\},
    \]
    where again $f(x)\in U_n$ is a definable relation on~$x$ and~$n$. 

    \smallskip
    
    The lemma is well-known. The $f$-preimage of a closed ball by a continuous function is closed, and an open ball is the union of closed balls. Hence
     \[
     f(x)\in B(y,r) \,\,\Leftrightarrow\,\, (\exists s<r)(\exists n)(\forall m\ge n) \, d(f_n(x),q) \le s. \qedhere.
     \]
    
\end{proof}

\begin{remark} \label{rmk:Borel_class_2}
    In fact, classically, a function is Baire class 1 if and only if the pull back of open sets is $F_\s$. The other direction requires a little bit of work; it is like Shoenfield's limit lemma, but for arbitrary spaces requires some topological considerations. 

    In second-order arithmetic, to make sense of the other direction, we would need to define codes for functions which pull back open sets to~$F_\s$ sets, similar to how continuous functions are coded by collection of pairs indicating that the pull back of an open set contains some open set. For such a code, we can computably construct a Baire class 1 code for the function, but the argument seems to require $\Sigma^0_2$ induction. 

    However, even stating the consistency of such a ``Borel class 2'' code seems complicated, as the containment relation between $F_\s$ sets is $\Pi^1_1$ complete. 
\end{remark}

\begin{proposition}[($\ACA$)] \label{prop:ACA_implies_CL1}
    Cousin's lemma holds for Baire class 1 gauges.
\end{proposition}

\begin{proof}
    The argument of \cref{lem:Cantor_to_UI} holds for Baire class 1 functions, and so it suffices to work in Cantor space. $\delta\colon 2^\w\to \R^+$ be a Baire class 1 gauge. By $\ACA$, for all~$x$, $\delta(x)$ exists. 

    Now the main two points are:
    \begin{enumerate}
        \item In $\ACA$, we can tell which $F_\s$ subsets of~$2^\w$ are empty;
        \item In $\ACA$, given a sequence $\seq{H_n}$ of nonempty $F_\s$ subsets of~$2^\w$, there is a choice sequence $\seq{x_n}$ so that $x_n\in H_n$ for all~$n$. 
    \end{enumerate}
    The point is that whether a tree has a path is an arithmetic question, and so whether the union of a countable sequence of closed subsets of Cantor space is empty or not is also arithmetic. For~(2), using~$\ACA$ we can obtain a sequence $\seq{S_n}$ of binary trees, each of which has a path, such that $[S_n]\subseteq H_n$, and then apply $\WKL$ (or in $\ACA$, just take the leftmost path of each~$S_n$).

    \medskip
    
    We can therefore repeat the proof of \cref{prop:Cousin_for_Borel}. Let $\delta$ be a Baire class 1 gauge on Cantor space. We define the set~$G$ and the tree~$T$ in exactly the same way; $G$ and~$T$ exists since they are arithmetically definable: $\delta(x) < 2^{-|\s|+1}$ is~$\Sigma^0_2$ by~\cref{lem:Sigma_2_preimages}. The same argument shows that~$T$ cannot have a path, and so by~$\WKL$, $T$ is finite. As above, from~$T$ we obtain a finite subset~$R$ of~$G$ which covers all of Cantor space. By~(2), we can find a sequence $\seq{x_\s\,:\,\s\in R}$ as required. 
\end{proof}

In the other direction:

\begin{theorem}[($\RCA$)]\label{thm:clb1->aca}
    Cousin's lemma for Baire class 1 functions implies $\ACA$.
\end{theorem}

The idea of the proof is as follows. Supposing we have a Cauchy sequence $\seq{z_n}$ with no limit. The sequence of functions $x \mapsto \abs{x-z_n}$ determine a Baire class 1 function~$\delta$, which is a gauge since $\seq{z_n}$ has no limit. But there cannot be any $\delta$-fine partition, since no such partition $P$ can cover the gap where $\lim z_n$ should be. Here are the details.

\begin{proof}[Proof of Theorem \ref{thm:clb1->aca}]
    If $\ACA$ fails then there is an increasing Cauchy sequence $\seq{z_n}_{n \in \N}$ of real numbers in the unit interval that has no limit. For each $n \in \N$, define 
     \[
     \delta_n(x) = \abs{x-z_n}.
     \]
     This is a sequence of continuous functions. This sequence is pointwise Cauchy: for all~$x$,~$n$ and~$m$, 
         \begin{equation*}
        \abs{\delta_m(x) - \delta_n(x)}\ = \ \big|\abs{x-z_m} - \abs{x-z_n} \big| \ \leq\ \abs{z_n - z_m};
    \end{equation*}
    so we use the fact that $\seq{z_n}$ is Cauchy. Note that in the ``real world'', $\seq{\delta_n}$ converges uniformly to the continuous function $x\mapsto  |x-z^*|$ where $z^* = \lim_n z_n$, but since $z^*$ does not exist in the model, this function does not have a code in the model. However in $\RCA$ we have just shown that $\seq{\delta_n}$ is a code for a Baire class 1 function, which we call~$\delta$. 

    \smallskip
    
    We claim that~$\delta$ is a gauge. For all $x\in [0,1]$, since $x\ne \lim_n z_n$, and since $\seq{z_n}$ is Cauchy, there is some $\epsilon>0$ such that for all but finitely many~$n$, $|x-z_n|\ge \epsilon$. Then $\delta(x) \ge \epsilon$.

    Now by \cref{lem:delta_fine_partition_and_cover} and Cousin's lemma applied to the gaugue $\tfrac12\delta$, there is a $\delta$-fine tagged partition $0=x_0 \le \xi_0 \le x_1 \le \xi_1 \le x_2 \le \dots \le \xi_{n-1} \le x_n = 1$; recall that this means that $x_{i+1}-x_i \le \delta(x_i)$ for all $i<n$.

    We derive a contradiction. By $\Sigma^0_1$-induction, there is a least~$i$ such that for all~$n$, $z_n \le x_{i+1}$. Since $\seq{z_n}$ is increasing, for all but finitely many~$n$, $z_n\in [x_{i},x_{i+1}]$. Since none of $x_{i}$, $\xi_{i}$ or $x_{i+1}$ are the limit of $\seq{z_n}$, there is some $\epsilon>0$ such that one of two happens:
    \begin{itemize}
        \item for all but finitely many~$n$, $z_n \in [x_{i}+\epsilon,\xi_{i}-\epsilon]$; or
        \item for all but finitely many~$n$, $z_n \in [\xi_{i}+\epsilon,x_{i+1}-\epsilon]$.
    \end{itemize}
    Without loss of generality, assume the former. Then for all but finitely many~$n$, 
     \[
     \delta_n(\xi_{i}) = \xi_i - z_n \le (\xi_i -x_i)-\epsilon \le (x_{i+1}-x_{i})-\epsilon,
     \]
    whence $\delta(\xi_{i})< x_{i+1}-x_{i}$, contradcting our assumption. The argument in the other case is the same. 
\end{proof}


\section{Higher Baire classes} 
\label{sec:higher_baire_classes}

Baire class 1 functions aren't closed under taking pointwise limits. Indeed, by iterating the operation of taking pointwise limits, we obtain a transfinite hierarchy of functions, which exhausts all Borel functions. 

In second-order arithmetic, recall that in $\RCA$, for a Baire class 1 function~$f$, the relation $f(x)\in B$ for open or closed balls~$B$ is definable even if $f(x)$ does not exist. We can therefore iterate. By (external) induction on standard~$n$ we define:

\begin{definition} \label{def:Baire_class_n}
    For each $n\ge 1$, a code for a Baire class~$(n+1)$ function~$f$ is a sequence $\seq{f_n}$ of codes for Baire class~$n$ functions such that for all~$x$, $\seq{f_n(x)}$ is Cauchy, in the sense that for all~$\epsilon>0$ there is an open ball~$B$ of radius $<\epsilon$ such that for all but finitely many~$n$, $f_n(x)\in B$. 

    If~$\seq{f_n}$ is such a code, then for any open ball $B = B(y,r)$, we say that $f(x)\in B$ if for some $s<r$, for all but finitely many~$n$, $f_n(x)\in B(y,s)$. 
\end{definition}

Of course by taking constant sequences, we see that every Baire class~$n$ function is also Baire class~$n+1$, so the classes are increasing. 

\smallskip

$\ACA$ implies, for each~$n$, that if $f$ is Baire class~$n$ then $f(x)$ exists for all~$x$. Similarly to \cref{lem:Sigma_2_preimages}, the inverse images of open sets by Baire class~$n$ functions are $\Sigma^0_{1+n}$. We can further extend the definition to Baire class~$\alpha$ functions for ordinals~$\alpha$; a code in this case is an $\alpha$-ranked well-founded tree, where each non-leaf has full splitting, the leaves are labeled by continuous functions, and each non-leaf node represents the pointwise limit of the functions represented by its children. The relation $f(x)\in U$ is then defined by transfinite recursion on the rank of a node, and so it makes most sense to use $\ATR$ as a base system for this kind of development; we do not pursue it here. We remark that $\ATR$ implies that every Borel function is Baire class~$\alpha$ for some~$\alpha$. 

\medskip

Now the proof of \cref{prop:ACA_implies_CL1} cannot be replicated for Baire class~$n$ functions for any for $n>1$. While $\ACA$ suffices to determine if a given $\Sigma^0_2$ subset of Cantor space is empty or not, it is $\Pi^1_1$-complete to determine whether a given~$\Pi^0_2$ set is empty. And indeed we show that Cousin's lemma for Baire class 2 functions is much stronger than $\ACA$: it implies $\ATR$. 

It would again be more convenient to work in Cantor space, and so we need a converse of \cref{lem:Cantor_to_UI}.

\begin{lemma}[($\RCA$)] \label{lem:UI_to_Cantor}
    Cousin's lemma for Baire class 2 functions on unit interval implies Cousin's lemma for Baire class 2 functions on Cantor space. 
\end{lemma}

Of course there is nothing special for the case $n=2$, the lemma holds for all Baire classes.

\begin{proof}
    We use the standard embedding of Cantor space into the unit interval: for $x\in 2^\w$ let
    \[
        \psi(x) = \sum_n 2x(n)3^{-n-1},
    \]
    and let $C = \psi[2^\w]$ be the ternary Cantor set. We use the following facts: $\psi$ is continuous; the function $z\mapsto d(z,C)$ is computable on $[0,1]$; and: for $x,y\in 2^\w$, if $d(x,y)= 2^{-n}$ then $|\psi(x)-\psi(y)| \ge 3^{-n}$, i.e.,
    \[
        |\psi(x)-\psi(y)|  \ge 3^{\log_2 d(x,y)}
    \]
    For all $x,y\in 2^\w$. 

    \smallskip
    
    Let $\delta$ be a Baire class 2 gauge on Cantor space. Define $\hat \delta \colon [0,1]\to \R$ by letting, for $z\in [0,1]$,
    \[
        \hat\delta(z) = \begin{cases*}
            3^{\log_2 \psi^{-1}(z)}, & if $z\in C$; and \\
            d(x,C), & otherwise. 
        \end{cases*}
    \]
    Then $\hat \delta$ is positive on~$[0,1]$, and it is Baire class 2; we use the fact that~$1_C$ (the characteristic function of~$C$) is Baire class~1, indeed with a very nice, uniformly computable approximation $\seq{g_n}$ of piecewise linear functions with $g_n(z)=1$ for all~$n$ when $z\in C$, and $g_n(z)=0$ for all but finitely many~$n$ when $z\notin C$. 

    \smallskip
    
    Suppose that $(P,\bar r_p)$ is a $\hat \delta$-fine cover. Define $Q = \left\{ x\in 2^\w \,:\,  \psi(x)\in P \right\}$ and for $x\in Q$ let $s_x = 2^{\log_3 r_{\psi(x)}}$. Then $(Q,\bar s_x)$ is a $\delta$-fine cover. The main point is that $z\in C$ cannot be covered by any $p\in P\setminus Q$, as for such~$p$ we have $\delta(p) = d(p,C) \le |p-z|$. 
\end{proof}

Before we give the reversal, we motivate our technique by showing something weaker: that $\ACA$ does not imply Cousin's lemma for Baire class 2 functions. We will show that every $\w$-model of $\ACA$ must contain $0^{(\w)}$ (in fact $0^{(\alpha)}$ for any computable ordinal~$\alpha$), in particular Cousin's lemma for Baire class 2 functions fails in the model consisting of the arithmetic sets. 

We review some definitions. The Turing jump $x'$ of $x\in 2^\w$ is the complete $\Sigma^0_1(x)$ set, identified with an element of $2^\w$. As mentioned above, the function $x\mapsto x'$ is Baire class 1, but more importantly, its graph (the relation $y = x'$) is $\Pi^0_2$. We will use the following well-known fact, which is provable in $\ACA$:

\begin{lemma}[($\ACA$)] \label{lem:double_jump_is_Baire_class_2}
    Let $f\colon 2^\w\to \Nat$ be a function. The following are equivalent:
    \begin{enumerate}
        \item[(1)] $f$ is $\Delta^0_3$-definable;
        \item[(2)] There is a partial computable function $\vphi\colon 2^\w\to \Nat$ such that for all~$x$, $f(x) = \vphi(x'')$; 
        \item[(3)] $f$ is Baire class 2. 
    \end{enumerate}
\end{lemma}

As a result we get:

\begin{lemma}[($\RCA$)] \label{lem:double_jump_and_Cousin}
    Suppose that $f\colon 2^\w\to \Nat$ is $\Delta^0_3$. Suppose that Cousin's lemma for Baire class 2 functions holds. Then there is a finite $P\subset 2^\w$ such that 
    \[
        2^\w = \bigcup_{p\in P} [p\rest{f(p)}].
    \]
\end{lemma}

\begin{proof}
    Define $\delta(x) = 2^{-f(x)}$. Then $\delta$ is $\Delta^0_3$. By \cref{lem:double_jump_is_Baire_class_2} (note that we already know that Cousin's lemma for Baire class 2 functions implies $\ACA$), $\delta$ is a Baire class 2 gauge. A $\delta$-fine cover $P\subset 2^\w$ is as required. 
\end{proof}

Recall the definition of a transfinite iteration of the Turing jump. Suppose that $\alpha$ is a linear ordering of a subset of~$\Nat$, which will usually be well-founded. An iteration of the Turing jump along~$\alpha$ is a set $H\subseteq \alpha\times \Nat$ satisfying, for all $\beta<\alpha$, 
\[
    H^{[\beta]} = \left(H^{[<\beta]}  \right)',
\]
where $H^{[\beta]} = \left\{ k \,:\,  (\beta,k)\in H \right\}$ and $H^{[<\beta]} = \left\{ (\gamma,k)\in H \,:\,  \gamma<\beta \right\}$. Again note that we view the domain of~$\alpha$ as a subset of~$\Nat$, so both $H^{[\beta]}$ and $H^{[<\beta]}$ are identified with elements of Cantor space, and so taking the Turing jump makes sense. 

If~$\alpha$ is indeed an ordinal (is well-founded), then $\ACA$ implies there is at most one iteration of the Turing jump along~$\alpha$, and if~$H$ is such, we write $H = 0^{(\alpha)}$. When this is relativised to an oracle~$x$, we write $H = x^{(\alpha)}$. The relation ``$H$ is an iteration of the Turing jump along~$\alpha$'' is $\Pi^0_2$. $\ATR$ is equivalent to the statement that for all ordinals~$\alpha$ and all~$x$, $x^{(\alpha)}$ exists. 

\medskip

Now let $\+M$ be an $\w$-model of~$\ACA$, let $\alpha$ be a computable ordinal, and suppose that $0^{(\w)}\notin \+M$; we show that Cousin's lemma for Baire class 2 functions fails in~$\+M$. Let $X\in 2^\w$. If $X\ne 0^{(\w)}$, there is some $n<\w$ such that 
\[
    X^{[n]} \ne \big(X^{[<n]}\big)';
\]
let $n(X)$ be the least such~$n$. The function $X\mapsto n(X)$ (as well as its domain) is $\Delta^0_3$. Further, for such~$X$, we let $k(X)$ be the least~$k$ such that 
\[
    X^{[n(X)]}(k) \ne \big(X^{[<n(X)]}\big)'(k);
\]
the function $X\mapsto k(X)$ is $\Delta^0_3$ as well. By their definition, for all $X\ne 0^{(\w)}$, 
\[
    X(n(X),k(X))\ne 0^{(\w)}(n(X),k(X)), 
\]
as $X^{[<n(X)]} = (0^{(\w)})^{[<n(X)]}$. Hence, we let $f(X) = \seq{n(X),k(X)}+1$; for all $X\ne 0^{(\w)}$, 
\[
    0^{(\w)} \notin [X\rest{f(X)}]. 
\]
Since $0^{(\w)}\notin \+M$, $f$ is total on $2^\w\cap \+M$, and since~$\+M$ is an $\w$-model, the same $\Delta^0_3$ definition holds in~$\+M$, so $\+M$ believes that $f$ is a Baire class 2 gauge on Cantor space. Let $P\subset 2^\w\cap \+M$ be finite. Since $0^{(\w)}\notin P$, we have 
\[
    0^{(\w)} \notin \bigcup_{X\in P} [X\rest{f(X)}],
\]
and so 
\[
    2^\w \ne \bigcup_{X\in P} [X\rest{f(X)}].
\]
But since~$P$ is finite, the latter fact is absolute for~$\+M$. Hence \cref{lem:double_jump_and_Cousin} fails in~$\+M$. 

\smallskip

Exactly the same argument shows that for any computable ordinal~$\alpha$, $0^{(\alpha)}$ is an element of every $\w$-model of Cousin's lemma for Baire class 2 functions. The more general argument below builds on this technique.

\begin{theorem}[($\RCA$)] \label{thm:main}
    Cousin's lemma for Baire class 2 functions implies $\ATR$. 
\end{theorem}

\begin{proof}
    By \cref{thm:clb1->aca}, we may work in~$\ACA$. 

    \smallskip
    
    Let $\alpha$ be an ordinal and let $x\in 2^\w$; and suppose that $x^{(\alpha)}$ does not exist, i.e., there is no iteration of the Turing jump (relativised to~$x$) along~$\alpha$. For simplicity of notation, suppose $x = 0$. 

    \smallskip
    
    Let 
    \[
        \+I = \left\{ \beta<\alpha \,:\,  0^{(\beta)}\,\text{exists} \right\}. 
    \]
    The set~$\+I$ does not exist (as a set in the model), but is definable. Again note that since $\ACA$ holds, for all $\beta\in \+I$ there is exactly one iteration of the Turing jump along~$\beta$, which is the set we call $0^{(\beta)}$. We observe that~$\+I$ cannot have a greatest element; if $0^{\beta}$ exists then by $\ACA$, so does $0^{(\beta+1)}$. 

    For any $X\in 2^\w$, since $X$ is not an iteration of the Turing jump along~$\alpha$, by arithemtic comprehension, there is a least $\beta<\alpha$ such that $X^{[\beta]}\ne (X^{[<\beta]})'$; call this $\beta(X)$. Then $X^{[<\beta(X)]} = 0^{(\beta(X))}$, so $\beta(X)\in \+I$. Similar to the argument above, we also let $k(X)$ be the least~$k$ such that 
    \[
        X^{[\beta(X)]}(k) \ne \big(X^{[<\beta(X)]}\big)'(k).
    \]
    From $k(X)$ and~$\beta(X)$ we can compute some $f(X)\in \Nat$ such that for all $\gamma \ge \beta(X)$ in $\+I$, 
    \[
        0^{(\gamma)} \notin [X\rest{f(X)}]. 
    \]
    Let $P\subset 2^\w$ be finite. Then 
    \[
        \beta^* = \max \{ \beta(X)\,:\, X\in P\}
    \]
    is an element of $\+I$. Since $\+I$ does not have a last element, take any $\gamma>\beta^*$ in~$\+I$; then 
    \[
        0^{(\gamma)} \notin \bigcup_{X\in P} [X\rest{f(P)}]. 
    \]
    Therefore, \cref{lem:double_jump_and_Cousin} fails, and so Cousin's lemma for Baire class 2 functions does not hold. 
\end{proof}



\end{document}